\documentclass[11pt]{article}
\usepackage{amssymb,amsmath,amsthm,hyperref,graphicx,color,array,fullpage,mathrsfs}
\usepackage[notref,notcite,color,final]{showkeys}
\usepackage{amscd}
\usepackage{makeidx}
\usepackage[OT2,T1]{fontenc}
\DeclareSymbolFont{cyrletters}{OT2}{wncyr}{m}{n}
\DeclareMathSymbol{\Sha}{\mathalpha}{cyrletters}{"58}
\makeindex

\newtheorem{theorem}{Theorem}[section]
\newtheorem{lemma}[theorem]{Lemma}
\newtheorem{corollary}[theorem]{Corollary}

\newtheorem{definition}[theorem]{Definition}

\newtheorem{proposition}[theorem]{Proposition}
\theoremstyle{remark}
\newtheorem{remark}[theorem]{Remark}

\begin{document}
\title{\textsc{Iwasawa Main Conjecture for Heegner Points: Supersingular Case }}

\author{\textsc{Francesc Castella and Xin Wan}}
\date{}

\maketitle
\begin{abstract}
In this paper we propose and prove an anticyclotomic Iwasawa main conjecture for Heegner points on supersingular elliptic curves such that $a_p=0$. The result has a ``$\pm$'' nature in the sense of Kobayashi. The proof uses a recent work of the second author on one divisibility of Iwasawa-Greenberg main conjecture for Rankin-Selberg $p$-adic $L$-functions, and an argument of Howard (adapted to our ``$\pm$''-situation). As a byproduct we also prove an improvement of C.Skinner's result on a converse of Gross-Zagier-Kolyvagin theorem.
\end{abstract}
\section{Introduction}
\subsection{Assumptions}
Let $p$ be an odd prime and $E$ an elliptic curve over $\mathbb{Q}$ having square-free conductor $N$ and supersingular reduction at $p$. By work of Wiles \cite{Wiles95} there is a cuspidal newform $f=\sum_{i=1}^\infty a_nq^n$ associated to $E$. Let $A:=E[p^\infty]$ and $T$ being its Tate module. Let $V=T\otimes_{\mathbb{Z}_p}\mathbb{Q}_p$. Assume $a_p=0$. Let $\mathcal{K}$ be a quadratic imaginary field with absolute Galois group $G_\mathcal{K}$, such that $p$ is split and $\mathrm{Im}(G_\mathcal{K})=\mathrm{Aut}(T)\simeq \mathrm{GL}_2(\mathbb{Z}_p)$ (regarding $T$ as a Galois representation over $G_\mathcal{K}$). We write $\epsilon$ for the global root number of $E$ over $\mathbb{Q}$. As in \cite{WAN} we make either of the following assumptions
\begin{itemize}
\item[(1)] there is at least one prime $q|N$ non-split in $\mathcal{K}$ and that the global root number of $E$ over $\mathcal{K}$ is $-1$;
\item[(2)] for each $q|N$ either $q$ is split in $\mathcal{K}$, or $q$ is ramified and $E$ has non-split multiplicative reduction at $q$, and suppose we have at least one such ramified prime. This implies that the root number of $E$ over $\mathcal{K}$ is $-1$.
\end{itemize}
We define $\mathcal{K}_\infty$ as the unique $\mathbb{Z}_p^2$-extension of $\mathcal{K}$ unramified outside $p$ and let $\Gamma_\mathcal{K}$ be the Galois group. Let $\Gamma^{\pm}$ be the subgroup such that complex conjugation acts by $\pm1$ and $\gamma^\pm$ be their topological generators. Let $\mathcal{K}^{\pm}_\infty$ be the fixed fields for $\Gamma^{\mp}$. Let $\Lambda_\mathcal{K}=\mathbb{Z}_p[[\Gamma_\mathcal{K}]]$ and $\Lambda=\mathbb{Z}_p[[\Gamma^-]]$. We define a $\mathbb{Z}_p$-automorphism $\iota$ of $\Lambda$ by sending $\gamma^-$ to $\frac{1}{\gamma^-}$. Let $\Psi$ be the character $\Gamma_\mathcal{K}\hookrightarrow\Lambda_\mathcal{K}^\times$ or $\Gamma\hookrightarrow \Lambda^\times$ and $\varepsilon$ being $\Psi\circ \mathrm{rec}$ for $\mathrm{rec}$ being the reciprocity map in class field theory (normalized by the geometric Frobenius). Let $\mathbf{T}$ be $T\otimes_{\mathbb{Z}_p}\Lambda(\Psi)$ and $\mathbf{A}$ being $T\otimes\Lambda(-\Psi)\otimes_{\Lambda}\Lambda^*$ with $\Lambda^*$ the Pontryagin dual of $\mathbf{T}$. We fix an $\iota_p: \mathbb{C}\simeq \mathbb{C}_p$ and let $p=v\bar{v}$ in $\mathcal{K}$ with $v$ being the place determined by $\iota_p$.
\subsection{Selmer Groups}

\noindent\underline{Selmer Conditions}:\\
\noindent We define some Selmer conditions ($\mathcal{F}^+,\mathcal{F}_v^+,\mathcal{F}_{\bar{v}}^+, v, \bar{v}, \mathrm{str}^+)$ for the Galois cohomology of $\mathcal{K}$ or finite extensions of it inside $\mathcal{K}_\infty$ of $V$ (in the following we write for $\mathcal{K}$ for example). We are going to define $p$-adic ``$+$'' local conditions $E^+(L)$ for $p$-adic fields $L$ in the text.
Let:
$$(\mathcal{F}^+)_w=\left\{\begin{array}{ll}\mathrm{ker}\{H^1(\mathcal{K}_w,V)\rightarrow H^1(I_w,V)\}, &w\nmid p\\ E^+(\mathcal{K}_w)\otimes_{\mathbb{Z}_p}\mathbb{Q}_p& w|p\end{array}\right.$$
$$(\mathcal{F}_v^+)_w=\left\{\begin{array}{ll}\mathcal{F}_w^+, &w\not=\bar{v}\\0& w=\bar{v}\end{array}\right.$$
$$(\mathcal{F}_{\bar{v}}^+)_w=\left\{\begin{array}{ll}\mathcal{F}_w^+, &w\not={v}\\0& w={v}\end{array}\right.$$
$$(v)_w=\left\{\begin{array}{ll}\mathcal{F}_w^+, &w\nmid p\\0& w=\bar{v}\\ H^1(\mathcal{K}_w,V)& w=v\end{array}\right.$$
$$(\bar{v})_w=\left\{\begin{array}{ll}\mathcal{F}_w^+, &w\nmid p\\0& w={v}\\ H^1(\mathcal{K}_w,V)& w=\bar{v}\end{array}\right.$$
$$(str)_w^+=\left\{\begin{array}{ll}\mathcal{F}_w^+, &w\nmid p\\0& w|p\end{array}\right.$$
We define the local Selmer conditions at primes not dividing $p$ for the Galois cohomology groups for $T$ or $A$ to be the inverse image or inverse image of the corresponding Selmer conditions for $V$. At $w|p$ we define the local Selmer condition at $w$ for $A$ to be $E^+(\mathcal{K}_w)\otimes_{\mathbb{Z}_p}\mathbb{Q}_p/\mathbb{Z}_p$ and define $H^1_+(\mathcal{K}_v, T)$ to be the orthogonal of $E^+(\mathcal{K}_w)\otimes_{\mathbb{Z}_p}\mathbb{Q}_p$ under local Tate pairing. We also define the local Selmer conditions for $\mathbf{T}$ and $\mathbf{A}$ at primes above $p$ by regarding them as projective or injective limits and taking the images in the limits.
We define the corresponding Selmer groups $H^1_{\mathcal{F}^+}(\mathcal{K},-)$ to be the inverse image in $H^1(\mathcal{K},-)$ of $\prod_w \mathcal{F}_w^+$ under the localization map and $$X^+:=H^1_{\mathcal{F}^+}(\mathcal{K},\mathbf{A})^*.$$ We sometimes write $H^1_+$ for $H^1_{\mathcal{F}^+}$. We also define Selmer groups for other Selmer conditions in a similar way. We often write $*$ for the orthogonal complement under local Tate pairing (for example for $\mathcal{F}^+$ we write $\mathcal{F}^{+*}$ for its orthogonal complement). Finally we write $\mathrm{rel}$ for the ``relaxed Selmer condition'' which is the same as $\mathcal{F}^+$ outside $p$ and puts no restriction at primes dividing $p$. We sometimes omit the subscript $\mathrm{rel}$ when writing the corresponding Selmer groups.\\

\subsection{Main Result}
In this paper we only prove the ``$+$'' main conjectures for simplicity. The ``$-$'' one can be proved in completely same way. We are going to construct a family of Heegner points $\kappa^+_1\in H^1_{\mathcal{F}^+}(\mathcal{K}, \mathbf{T})$ in the next section. Our main theorem is
\begin{theorem}\label{main Theorem}
There is a quasi-isomorphism $X^+\sim \Lambda \oplus X^+_{\mathrm{tor}}$, where $X_{\mathrm{tor}}^+$ is the torsion part of $X^+$. Moreover
\begin{itemize}
\item[(1)]
Under assumption (2) then for any height one prime $P$ of $\Lambda$,
$$\mathrm{lg}_PX_\mathrm{tor}^+=2\mathrm{lg}_P(\frac{H^1_+(\mathcal{K}, T\otimes\Lambda(\Psi))}{\Lambda\kappa_1^+}).$$
Here $\mathrm{lg}_P(M)$ means the length of $M_P$ as a $\Lambda_P$ module.
\item[(2)]
Under assumption (1) the above is true for all height one primes $P\not=(p)$.
\end{itemize}
\end{theorem}

\noindent Our theorem is proved by studying the relations between two different kinds of Selmer groups via Heegner points and their explicit reciprocity laws. These results were previously obtained in \cite{Howard} and \cite{WAN} for ordinary elliptic curves. In the supersingular case the arguments are more complicated since the local theory are of a quite different nature (actually a $\pm$-theory as in \cite{Kobayashi}). Such idea is also used by the second author to prove the cyclotomic main conjecture of Kobayashi in the supersingular case via explicit reciprocity law for Beilinson-Flach elements. We also remark that all these $\pm$-main conjectures are quite difficult to attack directly (especially the lower bounds for Selmer groups. See \cite{WAN1} for details). While the paper is being written, we learned that Longo and Vigni \cite{LV} worked on similar directions and they proved the $\pm$ dual Selmer module is rank one over $\Lambda$. The method they employed is different from ours, and we do not assume $\mathcal{K}_\infty^-/\mathcal{K}$ is totally ramified at $p$ as in \emph{loc.cit}.\\

\noindent We also prove a stronger form of the converse of theorem of Gross-Zagier and Kolyvagin by Skinner \cite{Skinner} as a corollary. It follows from Theorem \ref{main Theorem} in the same way as \cite[Theorem 1.7]{WAN}.
\begin{corollary}\label{Corollary}
Let $E/\mathbb{Q}$ be an elliptic curve with square-free conductor $N$ and supersingular reduction at $p$ with $a_p=0$. Let $\mathcal{K}$ be an imaginary quadratic field such that $p$ is split and such that $\mathrm{Im}(G_\mathcal{K})=\mathrm{Aut}(T)\simeq \mathrm{GL}_2(\mathbb{Z}_p)$. Suppose moreover assumption (1) is true. If the Selmer group $H_{\mathcal{F}}^1(\mathcal{K}, E[p^\infty])$ has corank one, then the Heegner point $\kappa_1$ is not torsion and thus the vanishing order of $L_\mathcal{K}(f,1)$ is exactly one.
\end{corollary}

\noindent This paper is organized as follows. In section 2 we introduce our $\pm$-local theory and construct the $\pm$- Heegner point Kolyvagin system in the sense of \cite{Howard} to give the upper bound for Selmer groups. In section 3 we first prove an explicit reciprocity law for Heenger points as is done in \cite{Castella} by the first author in the ordinary case. Then we use it to deduce the lower bound for Selmer groups from a two-variable Iwasawa-Greenberg main conjecture proved by the second author.
\section{Upper Bound for Selmer Groups}
\subsection{Some Local Theory}
\noindent We develop some local theory for studying the anticyclotomic Iwasawa theory in the supersingular case. These theories generalize the local theory in \cite{DI} (so in fact there is no need to assume the class number of $\mathcal{K}$ to be prime to $p$ in \emph{loc.cit} at least when $p$ is split in $\mathcal{K}$).

Define $\omega^+_n(X):=X\prod_{2\leq m\leq n, 2|m}\Phi_m(X)$ and $\omega^-_n(X):=\prod_{1\leq m\leq n, 2\nmid m}\Phi_m(X)$ (our definition is slightly different from \cite{Kobayashi}). It is easily seen that the prime $v|p$ is finitely decomposed as $v_1\cdots v_{p^t}$ in $\mathcal{K}^-$ (and thus is also decomposed into $p^t$ distinct primes in $\mathcal{K}_\infty$). Let $\Gamma_1$ ($\Gamma_1^-$) be the decomposition group of $v_1$ (and also other $v_i$'s) in $\Gamma_\mathcal{K}$ (or $\Gamma^-$). Let $H_\mathcal{K}$ be the Hilbert class field of $\mathcal{K}$ and $\mathcal{K}_0:=\mathcal{K}_\infty\cap H_\mathcal{K}$ and let $\mathcal{K}_m^-$ be the subfield of $\mathcal{K}_\infty$ such that $[\mathcal{K}_m^-:\mathcal{K}_0]=p^m$. Let $a$ be the inertial degree of $\mathcal{K}_0/\mathcal{K}$ at any $v|p$. Let $\mathbb{Q}_p^{ur}$ be the unramified $\mathbb{Z}_p$-extension of $\mathbb{Q}_p$ and $\mathbb{Q}_{p,\infty}$ being the cyclotomic $\mathbb{Z}_p$-extension of $\mathbb{Q}_p$. Let $\mathbb{Q}_{p,\infty}^{ur}$ be their composition. For $v|p$ we identify $\mathcal{K}_v\simeq \mathbb{Q}_p$. Let $u_v$ and $\gamma_v$ be topological generators of $U_v:=\mathrm{Gal}(\mathbb{Q}_{p,\infty}^{ur}/\mathbb{Q}_{p,\infty})$ and $\Gamma_v=\mathrm{Gal}(\mathbb{Q}_{p,\infty}^{ur}/\mathbb{Q}_{p}^{ur})$. In fact we choose $u_v$ to be the arithmetic Frobenius. Suppose $u_v$ and $\gamma_v$ are chosen such that $-p^au_v+\gamma_v$ is a topological generator for $\mathrm{Gal}(\mathcal{K}_{\infty,v}/\mathcal{K}_{\infty,v}^-)$. Let $X_v=\gamma_v-1$ and $Y_v=u_v-1$. We sometimes omit the subscript $v$ and write them $X$, $Y$, $\gamma$, $u$ later on. Let $T=\gamma^--1\in\mathbb{Z}_p[[\Gamma^-]]\simeq\mathbb{Z}_p[[T]]$.

For any unramified extension $k$ of $\mathbb{Q}_p$ we define
$$E^+[{k(\mu_{p^{n+1}})}]=\{x\in E({k(\mu_{p^{n+1}})})|\mathrm{tr}_{k(\mu_{p^{n+1}})/k(\mu_{p^{\ell+2}})}(x)\in E({k(\mu_{p^{\ell+1}})}), 0\leq \ell<n, 2|\ell\}.$$
Let $\hat{E}$ be the formal group for $E$. We also define the $+$-norm subgroup
$$\hat{E}^+[\mathfrak{m}_{k(\mu_{p^{n+1}})}]=\{x\in \hat{E}(\mathfrak{m}_{k(\mu_{p^{n+1}})})|\mathrm{tr}_{k(\mu_{p^{n+1}})/k(\mu_{p^{\ell+2}})}(x)\in \hat{E}(\mathfrak{m}_{k(\mu_{p^{\ell+1}})}), 0\leq \ell<n, 2|\ell\}.$$
Note that since $p\nmid\sharp E[\mathbb{F}_{p^m}]$, $$E^+[{k(\mu_{p^{n+1}})}]\otimes\mathbb{Q}_p/\mathbb{Z}_p=\hat{E}^+[\mathfrak{m}_{k(\mu_{p^{n+1}})}]\otimes
\mathbb{Q}_p/\mathbb{Z}_p.$$
As in \cite{WAN1}, for $z\in\mathcal{O}_k^\times$ we define a point $c_{n,z}\in \hat{E}[\mathfrak{m}_{k(\zeta_{p^n})}]$ such that
$$\log_{\hat{E}}(c_{n,z})=[\sum_{i=1}^\infty(-1)^{i-1}z^{\varphi^{-(n+2i)}}\cdot p^i]+\log_{f_z^{\varphi^{-n}}}(z^{\varphi^{-n}}\cdot(\zeta_{p^n}-1))$$
where $\varphi$ is the Frobenius on $k$ and $f_z(x):=(x+z)^p-z^p$, the $$\log_f(X)=\sum_{n=0}^\infty(-1)^n\frac{f^{(2n)}(X)}{p^n}$$
for $f^{(n)}=f^{\varphi^{n-1}}\circ f^{\varphi^{n-2}}\circ\cdots f(X)$. Let $k_n/k$ be the $\mathbb{Z}_p/p^n\mathbb{Z}_p$ sub-extension of $k(\mu_{p^{n+1}})$ and $\mathfrak{m}_{k,n}$ be the maximal ideal of its valuation ring. We also use the same notation $c_{n,z}$ for $\mathrm{tr}_{k(\zeta_{p^n})/k_{n-1}}c_{n,z}\in \mathfrak{m}_{k,n-1}$ as well. Let $k=k^m$ be unramified $\mathbb{Z}/p^m\mathbb{Z}$-extension of $\mathbb{Q}_p$. We sometimes write $k_{m,n}$ for the above defined $k_n$ with this $k=k^m$. Let $\Lambda_{n,m}=\mathbb{Z}_p[[\mathrm{Gal}(k_{n,m}/\mathbb{Q}_p)]]$. We have the following lemma.
\begin{lemma}
The $\mathbb{Z}_p$-module $E(k_{n,m})$ is torsion free. Moreover $E(k_{n,m})/p^{n'}E(k_{n,m})$ is its orthogonal complement under local Tate pairing
$$H^1(k_{n,m}, E[p^{n'}])\times H^1(k_{n,m}, E[p^{n'}])\rightarrow \mathbb{Z}_p/p^{n'}\mathbb{Z}_p.$$
\end{lemma}
\begin{proof}
It follows easily from the fact that $E[p]$ is an irreducible $G_{\mathbb{Q}_p}$-module.
\end{proof}

Choose $d:=\{d_m\}_m\in\varprojlim_{m}\mathcal{O}_{k^m}^\times$ where the transition is given by the trace map such that $d$ generates this inverse limit over $\mathbb{Z}_p[[U_v]]$ (such $d$ exists by the normal basis theorem). If $d_m=\sum_j a_{m,j}\zeta_j$ where $\zeta_j$ are roots of unity and $a_{m,j}\in\mathbb{Z}_p$. Define $c_{n,m}=\sum a_{m,j}c_{n,\zeta_j}$. Define $\Lambda_{n,m}^+:=\mathbb{Z}_p[[\Gamma_1]]/(\omega_n^+(X), (1+Y)^m-1)$. The following is \cite[Lemma 2.3]{WAN1}.
\begin{lemma}
For even $n$'s we have
$$\mathrm{tr}_{k_{m+1,n}/k_{m,n}}c_{n,m+1}=c_{n,m},$$
$$\mathrm{tr}_{k_{m,n}/k_{m,n-2}}c_{n,m}=-c_{n-2,m}$$
and that $c_{n,m}$ generates $\hat{E}^+[\mathfrak{m}_{k_{n,m}}]\simeq \Lambda_{n,m}^+$ as a $\Lambda_{n,m}^+$-module.
\end{lemma}

Define $\Gamma_{n,m}$ to be $\Gamma_1/(\gamma^n-1,u^m-1)$. Let $k_{n,m}$ be the subfield of $\mathbb{Q}_{p,\infty}^{ur}$ fixed by $\Gamma_{n,m}$. We define $E^+(k_{n,m})$ as the image of $E^+[{k(\mu_{p^{n+1}})}]$ under the trace map $\mathrm{tr}_{k(\mu_{p^{n+1}})/k_{n,m}}$.

\begin{definition}
Define $H^1_+(k_{n,m}, T)$ as the orthogonal complement of $E^+(k_{n,m})\otimes\mathbb{Q}_p/\mathbb{Z}_p$ under Tate local pairing.
Define $H^1_+(\mathcal{K}_v, T\otimes\mathbb{Z}_p[[\Gamma_1^-]](\Psi))$ to be the image of $H^1_+(\mathcal{K}_v, T\otimes\mathbb{Z}_p[[\Gamma_1]](\Psi))$ in $H^1(\mathcal{K}_v, T\otimes\mathbb{Z}_p[[\Gamma_1^-]](\Psi))$ using $\Gamma_1=\Gamma_1^-\times\Gamma_1^+$.
\end{definition}

As in \cite{WAN1} we may choose $b_{n,m}\in H^1_+(\mathcal{K}_v, T\otimes\Lambda_{n,m}(\Psi))$ such that $\varprojlim_{n,m}b_{n,m}$ generates $H^1_+(\mathcal{K}_v, T\otimes\mathbb{Z}_p[[\Gamma_1]](\Psi))$ as a free rank one module over $\mathbb{Z}_p[[\Gamma_1]]$ and $(-1)^{\frac{n+2}{2}}\omega_n^-(X)b_{n,m}=c_{n,m}$. We know $\mathcal{K}_{m,v}^-\subset k_{m,m+a}$. We take $$a_m:=\mathrm{tr}_{k_{m,m+a}/\mathcal{K}_{m,v}^-}b_{m,m+a}.$$ Then $a_m$ are norm compatible and $\varprojlim_m a_m$ generates $H^1_+(\mathcal{K}_v, T\otimes\mathbb{Z}_p[[\Gamma_1^-]])$ as a free $\mathbb{Z}_p[[\Gamma_1^-]]$-module. We have also seen in \cite{WAN1} that $H^1(\mathbb{Q}_p, T\otimes\mathbb{Z}_p[[\Gamma_1]])/H^1_+(\mathbb{Q}_p, T\otimes\mathbb{Z}_p[[\Gamma_1]])$ is free $\mathbb{Z}_p[[\Gamma_1]]$-module of rank one. Thus it is easy to see that $H^1(\mathbb{Q}_p, T\otimes\mathbb{Z}_p[[\Gamma_1^-]])/H^1_+(\mathbb{Q}_p, T\otimes\mathbb{Z}_p[[\Gamma_1^-]])$ is also free of rank one over $\mathbb{Z}_p[[\Gamma_1^-]]$.

We now define a local big logarithm map $\mathrm{LOG}^+$ on $H^1_+(\mathcal{K}_v, T\otimes\mathbb{Z}_p[[\Gamma_1^-]])$.
\begin{definition}
If $x=\varprojlim_nx_n\in H^1_+(\mathcal{K}_v, T\otimes\mathbb{Z}_p[[\Gamma_1^-]])$ such that $x_n=f_n\cdot a_n$ for $f_n\in\mathbb{Z}_p[[\Gamma_1^-/p^{n+a}\Gamma_1^-]]$ then it is easily seen that $\varprojlim_nf_n$ defines an element $f\in\Lambda$ which we define to be $\mathrm{LOG}^+(x)$.
Take $\gamma_1(=\mathrm{id}),\gamma_2,\cdots, \gamma_{p^t}$ be elements in $\Gamma_\mathcal{K}$ such that $\gamma_iv_1=v_i$. Then $\mathbb{Z}_p[[\Gamma^-]]=\oplus_{i=1}^{p^t}\gamma_i\mathbb{Z}_p[[\Gamma_1^-]]$. We also define $\mathrm{LOG}^+$ on $H^1_+(\mathcal{K}_v, T\otimes\Lambda(\Psi))=\oplus_iH^1_+(\mathcal{K}_v, T\otimes\mathbb{Z}_p[[\Gamma_1^-]](\Psi))\cdot\gamma_i$ by: if $x=\sum_i\gamma_ix_i$ then
$$\mathrm{LOG}^+x=\sum_i\gamma_i\cdot(\mathrm{LOG}^+x_i)\in\Lambda.$$
This does not depend on the choices of $\gamma_i$'s.
\end{definition}
We define
$$E^+(\mathcal{K}_{m,v})=\{x\in E(\mathcal{K}_{m,v})|\mathrm{tr}_{\mathcal{K}_{m,v}/\mathcal{K}_{\ell+1,v}}(x)\in E(\mathcal{K}_{\ell,v}), 0\leq \ell<m, 2|\ell\}.$$
Now let us give a description of this module. We have $$E^+(k_{m,m+a})\otimes\mathbb{Q}_p/\mathbb{Z}_p\simeq \Lambda_{m,m+a}^+\otimes\mathbb{Q}_p/\mathbb{Z}_p$$
using the fact that $c_{m,m+a}$ generates the norm subgroup over $\Lambda_{m,m+a}^+$. Under this identification we have $E^+(\mathcal{K}_{m,v})=\Lambda_{m,m+a}^+[\gamma'_v]$ for $\gamma'_v=-p^au_v+\gamma_v$. Let $Z:=(1+X)(1+Y)^{p^a}-1$. Moreover we have
\[\begin{CD}\frac{\mathbb{Z}_p[[X,Y]]}{(\omega_m^+(X), (1+X)(1+Y)^{-p^a}-1)}@>\times\frac{(1+Z)^{p^m}-1}{Z}>\sim>
\frac{(\frac{(1+Z)^{p^m}-1}{Z})\mathbb{Z}[[X,Y,Z]]}{((1+Z)^{p^m}-1,\omega_m^+(X), (1+X)(1+Y)^{-p^a}-(1+Z))},\end{CD}\]
where the latter term represents $E^+(\mathcal{K}_{m,v})$.
\[\begin{CD}\frac{\mathbb{Z}_p[[X,Y]]}{(\omega_m^+(X), (1+X)(1+Y)^{-p^a}-1)}@>>
\sim>\frac{\mathbb{Z}_p[[Y]]}{(\omega_m^+((1+Y)^{p^a}-1))}:=\Lambda_m^+.\end{CD}\]
This gives the $E^+(\mathcal{K}_{m,v})$ a $\Lambda_m$-module structure (isomorphic to $\Lambda_m^+$ defined above) under the natural projection $\mathrm{Gal}(\mathbb{Q}_{p,\infty}^{ur}/\mathbb{Q}_{p,\infty})\simeq \Gamma^-$.
Then one can easily checks that
$$E^+(\mathcal{K}_{m,v})\otimes\mathbb{Q}_p/\mathbb{Z}_p=(E^+(k_{m,m+a})\otimes\mathbb{Q}_p/\mathbb{Z}_p)\cap H^1(\mathcal{K}_{m,v}, E[p^\infty])$$
and so we have
\begin{lemma}
The $E^+(\mathcal{K}_{m,v})\otimes\mathbb{Q}_p/\mathbb{Z}_p$ is the orthogonal complement of $H^1_+(\mathcal{K}_{m,v}, T)$.
\end{lemma}

In the following we will often identify $u_v=1+Y$ with its image in $\Gamma^-$. We may choose $\gamma^-$ properly such that under the projection to $\mathrm{Gal}(\mathbb{Q}_{p,\infty}^{ur}/\mathbb{Q}_{p,\infty})\simeq \Gamma^-$ we have $(1+Y)=(1+T)^{p^t}$.
\begin{remark}
The local theory developed in \cite{DI} is a special case of our theory when $a=0$ and $t=0$.
\end{remark}
Take a $\bar{\mathbb{Q}}_p$-point $\phi$ of $\mathrm{Spec}\mathbb{Z}_p[[\Gamma_v\times U_v]]$ which corresponds to a finite order character $\chi_\phi$ of $\Gamma_v\times U_v$ mapping $\gamma_v$ and $u_v$ to primitive $p^n$ and $p^m$-th roots of unities, with $n$ being positive and even. Suppose $x=\varprojlim_{n,m}x_{n,m}$ and $\mathrm{LOG}^+(x)=f=\varprojlim_{n,m}f_{n,m}$. Then the following formulas are in \cite[Proposition 2.10]{WAN1} and will be useful.
\begin{equation}\label{Interpolation Formula}
\sum_{\tau\in\Gamma_v/p^n\Gamma_v\times U_v/p^mU_v} \log_{\hat{E}}x_{n,m}^\tau\cdot \chi_\phi(\tau)=\frac{\phi(f_{n,m})\sum \log_{\hat{E}}c_{n,m}^\tau\chi_\phi(\tau)}{\phi(\omega_n^-(X))}(-1)^{\frac{n+2}{2}},
\end{equation}
\begin{equation}\label{equation (2)}
\sum_{\tau\in\Gamma_v/p^n\Gamma_v\times U_v/p^mU_v}\log_{\hat{E}}(c_{n,m})^\tau\chi_\phi(\tau)=\mathfrak{g}(\chi_\phi|_{\Gamma})\cdot \chi_\phi(u)^n\cdot \sum_{u\in U_m}\chi_\phi(u)d_m^u.
\end{equation}

\noindent For $S$ a finite set of primes not dividing $p$ in $\mathcal{K}$ we also write $S$ for the product of the primes in $S$.
For each integer $m$ we denote $\mathcal{K}_m[S]$ for the composition field of $\mathcal{K}_m$ and the ring class field $\mathcal{K}[S]$ of conductor $S$.
\begin{lemma}\label{Lemma 2.2}
The module $M_S:=H^1(\mathcal{K}[S],T\otimes\Lambda(\Psi))$ is free over $\Lambda$.
\end{lemma}
\begin{proof}
We first show that $M_S$ is $\Lambda$ torsion-free
and
\[\begin{CD} M_S/XM_S@>\times p>> M_S/XM_S\end{CD}\] are injective.
The first follows easily from that $E[p]|_{G_\mathcal{K}}$ is irreducible. We also get $M_S/XM_S\hookrightarrow H^1(\mathcal{K}[S], T)$. So it suffices to prove that
\[\begin{CD}H^1(\mathcal{K}[S], T)@>\times p>>H^1(\mathcal{K}[S], T)\end{CD}\]
is injective. Again this follows from that $E[p]|_{G_\mathcal{K}}$ is irreducible.
From the first claim above and the structure theorem of $\Lambda$-modules we know that $M_S$ is injected to a free finitely generated $\Lambda$-module $N$ with finite cardinality. On the other hand $\mathrm{Tor}^1_\Lambda(N, \Lambda/X\Lambda)$ is a non-zero $\mathbb{Z}_p$-torsion module which is injected into $M_S/XM_S$, which contradicts the second claim above. We thus get $N=0$ and $M_S$ is free over $\Lambda$.
\end{proof}

\noindent\underline{Poitou-Tate exact sequences}\\
We record here the Poitou-Tate exact sequence which will be used later on. Let $\mathcal{F}\subseteq \mathcal{G}$ be two Selmer conditions then we have the following
\begin{proposition}
We have the following long exact sequence:
$$0\rightarrow H_\mathcal{F}^1(\mathcal{K},\mathbf{T})\rightarrow H_\mathcal{G}^1(\mathcal{K},\mathbf{T})\rightarrow H_{\mathcal{G}}^1(\mathcal{K}_v,\mathbf{T})/H_\mathcal{F}^1(\mathcal{K}_v,\mathbf{T})\rightarrow H_{\mathcal{F}^*}^1(\mathcal{K},\mathbf{A})^*\rightarrow H_{\mathcal{G}^*}^1(\mathcal{K},\mathbf{A})^*\rightarrow 0.$$
\end{proposition}
\begin{proof}
It follows from \cite[Theorem 1.7.3]{Rubin} in a similar way as corollary 1.7.5 in \emph{loc.cit}.
\end{proof}
\subsection{Heegner Points}
\subsubsection{Heegner Points Construction}
Let $B(E)$ be the set of odd rational primes $q$ not dividing $\mathrm{cond}(E)$ and inert in $\mathcal{K}$ and such that $n(q)=\mathrm{exp}_p(q+1,a_q)>0$. An important fact to keep in mind is that $q$ splits completely in $\mathcal{K}_\infty^-/\mathcal{K}$. Let $\Lambda^r$ be the set of the products of $r$ distinct members of $B(E)$ and $\Lambda_n^r$ be the set of products of $r$ distinct primes in $B(E)$ such that $n(q)\geq n$. Suppose $\mathbb{T}$ is a Galois representation of $G_\mathbb{Q}$ with coefficient ring $R$ a complete local ring (in practice a discrete valuation ring or an Iwasawa algebra). For any prime $q\in B(E)$ which generates a prime $v$ of $\mathcal{K}$ such that $I_v$ (the inertial group at $v$) acts trivially on $\mathbb{T}$, let $\mathbf{I}_q$ be the ideal of $R$ generated by $q+1$ and the Frobenius trace for $G_{\mathcal{K}_v}$ acting on $\mathbb{T}$. Define $$H_f^1(\mathcal{K}_v, \mathbb{T}/\mathbf{I}_q):=\mathrm{ker}\{H^1(\mathcal{K}_v, \mathbb{T}/\mathbf{I}_q)\rightarrow H_f^1(I_v, \mathbb{T}/\mathbf{I}_q)\}.$$
We have natural isomorphism of $R/\mathbf{I}_q$-modules,
$$H_f^{1}(\mathcal{K}_v, \mathbb{T}/\mathbf{I}_q)\simeq \mathbb{T}/\mathbf{I}_q.$$
Define $H_s^1(\mathcal{K}_v,\mathbb{T}/\mathbf{I}_q)=\frac{H^1(\mathcal{K}_v,\mathbb{T}/\mathbf{I}_q)}
{H_f^1(\mathcal{K}_v,\mathbb{T}/\mathbf{I}_q)}$.
Then
$$H_s^{1}(\mathcal{K}_v, \mathbb{T}/\mathbf{I}_q)\simeq \mathbb{T}/\mathbf{I}_q$$
as $R/I_q$-modules. These can be seen, for example by \cite[Lemma 1.2.1]{MR}.
There is also a ``finite-singular isomorphism'' $\phi_v^{fs}: H^1_f(\mathcal{K}, \mathbb{T}/\mathbf{I}_q)\rightarrow H^1_s(\mathcal{K},\mathbb{T}/\mathrm{I}_q)$. If $\mathbb{T}$ has coefficient in a discrete valuation ring with uniformizer $\varpi$ and $\varpi^n\in \mathbf{I}_q$. Then we can also define $H^1_f(\mathcal{K}, \mathbb{T}/\varpi^n\mathbb{T})$ and $H^1_s(\mathcal{K}, \mathbb{T}/\varpi^n\mathbb{T})$ similarly. In \cite[Page 8, 9]{Kolyvagin} Kolyvagin also defined a direct summand of $H^1(\mathcal{K}, \mathbb{T}/\varpi^n\mathbb{T})$ which maps
isomorphically to $H^1_s(\mathcal{K}, \mathbb{T}/\varpi^n\mathbb{T})$, which we still
denote as $H^1_s(\mathcal{K}, \mathbb{T}/\varpi^n\mathbb{T})$.\\

\noindent One can make the following constructions as in \cite[Sections 1.7, 2.3]{Howard} and \cite{Fouquet}. For each $S$ a square-free product of distinct primes in $B(E)$ and some $m\geq 0$, one can construct a class 
$$P_m[S]\in H^1(\mathcal{K}_m[S], T)$$
using the Kummer map of certain Heegner points. These satisfies the norm relation
$$\mathrm{Nm}_{\mathcal{K}_{m+2}/\mathcal{K}_{m+1}}P_{m+2}[S]=P_m[S]$$
as in \cite[Lemma 4.2]{DI}.
We also claim that the $P_m[S]$'s lie in the image of the nature map $H^1(K[S], T\otimes\Lambda(\Psi))\rightarrow H^1(K_m[S], T)$ (i.e. they come from universal norms). This can be seen easily as follows: from the obvious long exact sequence we see the cokernel of this image is $H^2(K[S], T\otimes\Lambda(\Psi))[\omega_m]$. This last expression stablizes as $m$ approaches infinity, since $H^2(K[S], T\otimes\Lambda(\Psi))$ is a finitely generated $\Lambda$ module and we have a structure theorem. On the other hand we have $\omega^+_{m_1}/\omega^+_{m_2}P_{m_1}[n]=P_{m_2}[n]$ for $m_1>m_2$. So taking $m_2$ to approach infinity, it is easy to see that any $P_m[S]$ must have zero image in the cokernel mentioned above.

\begin{remark}
The Heegner points are constructed from Shimura curves associated to certain quaternion algberas (possibly $\mathrm{GL}_2$). Such quaternion algebra is ramified exactly at the primes where the local sign of $E$ over $\mathcal{K}$ is $-1$. (In literature sometimes people use the local signs of Rankin-Selberg products which differs from our convention by a factor $\chi_{\mathcal{K}/\mathbb{Q},v}(-1)$, $\chi_{\mathcal{K}/\mathbb{Q}}$ being the quadratic character corresponding to $\mathcal{K}/\mathbb{Q}$.) Therefore this quaternion algebra is $\mathrm{GL}_2$ under our assumption (2) might be general under assumption (1).
\end{remark}

Now we construct the $+$-Heegner point classes. By the freeness result Lemma \ref{Lemma 2.2} and the above claim, there is an unique $P^+_{m}[S]\in\frac{H^1(\mathcal{K}_m[S], T\otimes\Lambda(\Psi))}{\omega_n^+((1+Y)^{p^a}-1)H^1(\mathcal{K}[S], T\otimes\Lambda(\Psi))}$ such that $\omega_n^-((1+Y)^{p^a}-1)P_{m}^+=(-1)^{\frac{m+2}{2}}P_m[S]$ and they form a norm-compatible system. We define $P^+[S]:=\varprojlim_mP^+_m[S]$ as an element of $H^1(\mathcal{K}[S], T\otimes\Lambda(\Psi))$. (Note that $\{\omega_n^+((1+Y)^{p^a}-1)\}_n$ forms a basis for the topology of $\Lambda$. So this indeed defines a class.). By definition, for $v|p, \mathrm{loc}_vP^+[S]\in H^1_+(\mathcal{K}_v,\mathbf{T})$. Now from the $P^+[S]$'s we can use the derivative construction in \cite[Section 1.7]{Howard} to get a Kolygavin system $\kappa_S^+$ for the Selmer structure $\mathcal{F}^+$ in the sense of \cite[Section 1.2]{Howard}, which satisfies the norm relation that if $q\in B(E)$, $q\nmid S$,
\begin{equation}\label{(1)}
\phi^{fs}_qv_q\kappa_S^+\equiv\partial_q\kappa_{Sq}^+(\mathrm{mod}\ \mathbf{I}_q).
\end{equation}
The following proposition is of crucial importance.
\begin{proposition}
The class $\kappa_1^+$ is non-torsion in $H^1_+(\mathcal{K}, \mathbf{T})$.
\end{proposition}
\begin{proof}
It is proved in \cite{CV} that this class is non-trivial. See Theorem 1.5 of \emph{loc.cit} and the discussion right after it. And then the proposition follows from Lemma \ref{Lemma 2.2}.
\end{proof}

Now let $P=P_\phi$ be a height one prime of $\Lambda$ which corresponds to a point $\phi\in\mathrm{Spec}\Lambda$. Let $\mathcal{O}=\mathcal{O}_\phi$ be the integral closure of $\Lambda/P$ and $\varpi_\phi$ being its uniformizer. Let $T_\phi$ be the Galois representation $\mathbf{T}\otimes_\Lambda\mathcal{O}_\phi$ and ${A}_\phi$ being $T_\phi\otimes_{\mathcal{O}_\phi}\mathcal{O}_\phi^*$. Define $T_{\phi,m}$ and $A_{\phi,m}$ as the truncation modulo $\varpi_\phi^m$.

We define some local Selmer conditions at primes dividing $p$. We have
$$H^1(\mathbb{Q}_p, \mathbf{T})\rightarrow H^1(\mathbb{Q}_p, \mathbf{T}\otimes_\Lambda\Lambda/P)\rightarrow H^1(\mathbb{Q}_p, T_\phi)$$
is the natural map
$$\Lambda\oplus\Lambda\rightarrow \Lambda/P\oplus\Lambda/P\rightarrow  \mathcal{O}_\phi\oplus\mathcal{O}_\phi$$
since each term above is generated by two elements $v_1$ and $v_2$ whose images generate $H^1(\mathbb{Q}_p, T\otimes\mathbb{F}_p)$ as $\mathbb{F}_p$-vector spaces. If $v_+$ is a generator of $H^1_+(\mathbb{Q}_p, \mathbf{T})$ over $\Lambda$, we define $H^1_+(\mathbb{Q}_p, T_\phi)$ to be the sub-$\mathcal{O}_\phi$ module of $H^1(\mathbb{Q}_p, T_\phi)$ generated by the image of $v_+$. We also make the similar definition for the truncated representation $T_{\phi,m}$.

We define
$H^1_+(\mathcal{K}_v, A_\phi)$ ($H^1_+(\mathcal{K}_v, A_{\phi,m})$) as the orthogonal complement of $H^1_+(\mathcal{K}_v, T_\phi)$ ($H^1_+(\mathcal{K}_v, T_{\phi,m})$, respectively) under Tate pairing.
\subsection{The Upper Bound}
The situation here can now be conveniently adapted to the framework established in \cite{Howard}. For the Galois representation $T_\phi/\varpi_\phi^mT_\phi$ we take the Selmer condition $\mathcal{F}$ in \emph{loc.cit} to be as there at primes outside $p$, and the $+$-Selmer conditions defined above at $p$-adic places. We need to check the axioms (H.0)-(H.5) in \cite[Section 1.3]{Howard}. The only non-trivial part is in (H.4) we need to check the $+$ Selmer condition at $p$ is the orthogonal complement of itself under the local Tate pairing. In fact we have the following proposition is an easy consequence of \cite[Proposition 3.14]{Kim1} (by taking the $d$ there to be a sufficiently large number $p^n$ so that we have inclusion of ideals of $\Lambda$
$$(\omega_n(X),p^m)\subseteq (P_\phi, p^m).)$$
\begin{proposition}
$H^1_+(\mathbb{Q}_p, T_{\phi,m})$ is the orthogonal complement of $H^1_+(\mathbb{Q}_p, T_{\phi^\iota,m})$ under the local Tate pairing.
\end{proposition}
From the proposition we see that $H^1_+(\mathbb{Q}_p, A_{\phi,m})$ can be identified with $H^1_+(\mathbb{Q}_p, T_{\phi,m})$. We use the $\kappa_S^+$'s constructed above as the Kolyvagin system over $\Lambda$ for $(T, \mathcal{F}^+,\mathcal{L})$ ($\mathcal{L}$ is a set of degree two primes of $\mathcal{K}$ as in \cite[Section 1.2]{Howard}) in the sense of
\cite[Section 2.2]{Howard} (whose specializations to $\phi$'s are Kolyvagin systems defined in Section 1.2 of \emph{loc.cit}). The Kolyvagin system machinery thus gives the upper bound for Selmer groups in Theorem \ref{main Theorem}. Indeed Howard's argument carries out throughout by noting that at a prime $v$ above $p$, $H^1_+(\mathcal{K}_v, A_\phi)$ is contained in the inverse image of $\varinjlim_m E^+(\mathcal{K}_{m,v})\otimes\mathbb{Q}_p/\mathbb{Z}_p$ under
$$H^1(\mathcal{K}_v, A_\phi)\rightarrow H^1(\mathcal{K}_v,\mathbf{A}[P])\rightarrow H^1(\mathcal{K}_v,\mathbf{A})[P],$$
and, for every point $\phi\not=(p)$, $H^1_+(\mathcal{K}_v, A_\phi)\rightarrow H^1_+(\mathcal{K}_v, \mathbf{A}[\phi])$ has kernel and cokernel of cardinality bounded depending only on the ring structure of $\Lambda/P_\phi$.

\section{Lower Bound for Selmer Groups}
\subsection{Two-Variable Main Conjecture}

We recall the two-variable main conjecture proved in \cite{WAN2}.
%Let $\hat{\mathbb{Z}}_p^{ur}$ be the completion of the maximal unramified extension of $\mathbb{Z}_p$.
%On the arithmetic side, the local Selmer condition in \cite{WAN2} is the $v$-one defined here. On the analytic side,
As shown in [\emph{loc.cit.}, \S{4.6}], there is a $p$-adic $L$-function
$\mathcal{L}_{f,\mathcal{K}}\in\mathrm{Frac}(\hat{\mathbb{Z}}_p^{ur}[[\Gamma_{\mathcal{K}}]])$
such that for a Zariski dense set of arithmetic points $\phi\in\mathrm{Spec}\Lambda_\mathcal{K}$,
with $\phi\circ \boldsymbol{\varepsilon}$ the $p$-adic avatar of a Hecke character $\xi_\phi$ of $\mathcal{K}^\times\backslash \mathbf{A}_\mathcal{K}^\times$ of infinite type $(\frac{\kappa}{2},-\frac{\kappa}{2})$ ($\kappa\geq 6$) and conductor $(p^t,p^t)$ ($t>0$) at $p$, we have
\[
\phi(\mathcal{L}_{{f},\mathcal{K}})=\frac{p^{(\kappa-3)t}\xi_{1,p}^2\chi_{1,p}^{-1}
\chi_{2,p}^{-1}(p^{-t})\mathfrak{g}(\xi_{1,p}\chi_{1,p}^{-1})
\mathfrak{g}(\xi_{1,p}\chi_{2,p}^{-1})L(\mathcal{K},f,\xi_\phi,\frac{\kappa}{2})(\kappa-1)!(\kappa-2)!\Omega_p^{2\kappa}}{(2\pi i)^{2\kappa-1}\Omega_\infty^{2\kappa}},
\]
where $\mathfrak{g}$ denotes the Gauss sum, $\chi_{1,p}$ and $\chi_{2,p}$ are characters such that $\pi(\chi_{1,p},\chi_{2,p})\simeq \pi_{f,p}$,
and $\Omega_\infty$ and is a CM period attached to $\mathcal{K}$ and $\Omega_p$ is the corresponding $p$-adic period.
%Recall we have chosen $d=\varprojlim_md_m\in\varprojlim \mathcal{O}_{k^m}^\times$ where the transition map is the trace map. We define %$F_d\in\hat{\mathbb{Z}}_p^{ur}[[U]]$ as the image of $\varprojlim_m\sum_{u\in U_v/p^mU_v}d_m^u\cdot u$.
This $p$-adic $L$-function can also be constructed by Hida's Rankin--Selberg method \cite{Hida91}.
%See \cite[Remark 7.2]{WAN3} for a detailed discussion.
In fact Hida's construction gives an element in $\mathrm{Frac}(\Lambda_\mathcal{K})$ and the above $\mathcal{L}_{f,\mathcal{K}}$ is obtained by multiplying Hida's by a Katz $p$-adic $L$-function $\mathcal{L}^{Katz}_\mathcal{K}\in\hat{\mathbb{Z}}_p^{ur}[[\Gamma_\mathcal{K}]]$. Let
$$F_{d,2}=\varprojlim_m\sum_{\sigma\in U/p^mU}d_m^\sigma\cdot\sigma^2.$$
Then the discussion in \cite[Section 6.4]{LZ} on Katz $p$-adic $L$-functions (see also the discussion in Section 3.2 of \emph{loc.cit}) implies that $\mathcal{L}^{Katz}_\mathcal{K}/F_{d,2}$ is actually an element in $\mathbb{Z}_p[[\Gamma_\mathcal{K}]]\backslash\{0\}$. (The square comes from that fact that $\chi\mapsto\chi\chi^{-c}$ induces square map on anticyclotomic characters). We have the following lemma which follows from straightforward computation.
\begin{lemma}
$$\varprojlim_m\sum_{\sigma\in U/p^mU}d_m^\sigma\cdot\sigma^2\in (\varprojlim_m\sum_{\sigma\in U/p^mU}d_m^\sigma\cdot\sigma)^2\cdot\mathbb{Z}_p[[U]]^\times.$$
\end{lemma}
So there is an $\mathcal{L}_{f,\mathcal{K}}'\in\mathrm{Frac}\Lambda_\mathcal{K}$ such that
$$\mathcal{L}'_{f,\mathcal{K}}\cdot F_d^2=\mathcal{L}_{f,\mathcal{K}}.$$
\noindent The two variable main conjecture states that:
$$\mathrm{char}_{\Lambda_\mathcal{K}}(X_{f,\mathcal{K},v})=(\mathcal{L}_{f,\mathcal{K}}')$$
Where $X_{f,\mathcal{K},v}:=H_v^1(\mathcal{K}, T\otimes \Lambda_\mathcal{K}(\Psi_\mathcal{K})\otimes_{\Lambda_\mathcal{K}}\Lambda_\mathcal{K}^*)^*$.
The following is the part (ii) of the main theorem of \cite{WAN1}.
\begin{proposition}\label{3.2}
Under assumption (1) we have one containment $$\mathrm{char}_{\Lambda_\mathcal{K}\otimes\mathbb{Q}_p}(X_{f,\mathcal{K},v})\subseteq(\mathcal{L}_{f,\mathcal{K}}')$$
as fractional ideals of $\Lambda_\mathcal{K}\otimes_{\mathbb{Z}_p}\mathbb{Q}_p$.
\end{proposition}
To apply the above proposition to our situation we need the following control result of Selmer groups. Let $X_{f,\mathcal{K},v}^{anti}$ be the anticyclotomic dual $v$-Selmer group defined as $X_{f,\mathcal{K},v}$ by replacing the Galois representation $T\otimes\Lambda_\mathcal{K}\otimes_{\Lambda_\mathcal{K}}\Lambda_\mathcal{K}^*$ by $T\otimes\Lambda\otimes_{\Lambda}\Lambda^*$. The proof is identical to \cite[Proposition 3.1]{WAN}.
\begin{proposition}\label{3.3}
Let $I$ be the ideal of $\Lambda_\mathcal{K}$ generated by $(\gamma^+-1)$. Under assumption (1) there is an isomorphism
$(X_{f,\mathcal{K},v}/IX_{f,\mathcal{K},v})\otimes_{\mathbb{Z}_p} \mathbb{Q}_p\simeq X_{f,\mathcal{K},v}^{anti}\otimes_{\mathbb{Z}_p} \mathbb{Q}_p$ of $\mathbb{Z}_p[[\Gamma_\mathcal{K}^-]]\otimes_{\mathbb{Z}_p} \mathbb{Q}_p$-modules. Under assumption (2) the above equality is true as $\Lambda\otimes_{\mathbb{Z}_p}\mathbb{Q}_p$-modules.
\end{proposition}
Write $\mathcal{L}^-_{f,\mathcal{K}}$ for the image of $\mathcal{L}_{f,\mathcal{K}}'$ under
$$\mathbb{Z}_p[[\Gamma^+\times\Gamma^-]]\rightarrow \mathbb{Z}_p[[\Gamma_\mathcal{K}^-]], \gamma^+\rightarrow 1.$$
Then it is an element of $\mathbb{Z}_p[[\Gamma^-]]$ (for example by looking at another construction of it using Heegner points \cite{Hsieh}).
We have the following corollary.
\begin{corollary}
Under assumption (1) we have one containment $$\mathrm{char}_{\Lambda\otimes_{\mathbb{Z}_p}\mathbb{Q}_p}(X_{f,\mathcal{K},v}^{anti})\subseteq(\mathcal{L}^-_{f,\mathcal{K}})$$
as fractional ideals of $\Lambda\otimes_{\mathbb{Z}_p}\mathbb{Q}_p$.
\end{corollary}

\subsection{Explicit Reciprocity Law for Heegner Points}
Recall that $\mathcal{L}_{f,\mathcal{K}}|_{\hat{\mathbb{Z}}_p^{ur}[[\Gamma^-]]}=\mathcal{L}_{f,\mathcal{K}}^-\cdot F_d^2$.
\begin{proposition}\label{3.1}
We have $$(\mathrm{LOG}^+_v(\kappa_1^+))^2=(\mathcal{L}_{f,\mathcal{K}}^-)$$
under assumption (2). This identity is true up to multiplying by a constant in $\bar{\mathbb{Q}}_p^\times$ under assumption (1).
\end{proposition}
\begin{proof}
Let $\psi$ be an anticyclotomic Hecke character of infinity type $(1,-1)$ and of conductor prime-to-$p$,
let $\mathscr{L}_{\mathfrak{p},\psi}(f)\in\mathbb{Z}_p^{ur}[[\Gamma^-]]$ be the $p$-adic $L$-function in \cite[Def.~3.5]{castella-hsieh},
and set
\[
\mathscr{L}_{\mathfrak{p}}(f):={\rm Tw}_{\psi^{-1}}(\mathscr{L}_{\mathfrak{p},\psi}(f)),
\]
where ${\rm Tw}_{\psi^{-1}}:\mathbb{Z}_p^{ur}[[\Gamma^-]]\rightarrow\mathbb{Z}_p^{ur}[[\Gamma^-]]$ is the $\mathbb{Z}_p^{ur}$-linear
isomorphism given by $\gamma\mapsto\psi^{-1}(\gamma)\gamma$ for $\gamma\in\Gamma^-$. Thus
$\phi\left(\mathscr{L}_{\mathfrak{p}}(f)\right)=\chi\left(\mathscr{L}_{\mathfrak{p}}(f)\right)$,
where $\chi:=\psi^{-1}\phi$, and comparing their interpolation properties one easily checks that
\[
\mathcal{L}_{f,\mathcal{K}}\vert_{\mathbb{Z}_p^{ur}[[\Gamma^-]]}=\mathscr{L}_{\mathfrak{p}}(f)^2
\]
up to units. Following the calculations in \cite[Thm.~4.8]{castella-hsieh}, we find at that if $\phi$ is any primitive
character of $\Gamma^-/p^m\Gamma^-$ with $n:=m-t-a$ even, then
\[
\phi\left(\mathscr{L}_{\mathfrak{p}}(f)\right)=
\mathfrak{g}(\chi_\phi^{-1})\chi_{\phi}(p^n)p^{-n}
\sum_{\sigma\in{\rm Gal}(H_{p^n}/K)}\chi(\sigma){\rm log}_{\omega_f}(z_{p^n}^\sigma).
\]
Combined with the definition of $\kappa_n^+$ and formulas $(1)$ and $(2)$, we see that
\begin{align*}
\phi\left(\mathscr{L}_{\mathfrak{p}}(f)\right)&=
\mathfrak{g}(\chi_\phi^{-1})\chi_{\phi}(p^n)p^{-n}
\phi(\omega_n^-(X))\sum_{\sigma\in{\rm Gal}(H_{p^n}/K)}\chi(\sigma){\rm log}_{\hat{E}}((\kappa_n^+)^\sigma)(-1)^{n/2}\\
&=\pm\chi_\phi(p^n)\mathfrak{g}(\chi_\phi)^{-1}
\sum_{\sigma\in{\rm Gal}(H_{p^n}/K)}\chi(\sigma)\log_{\hat{E}}(c_n^\sigma)\cdot\phi\left({\rm Log}^+(\kappa_1^+)\right)\\
&=\pm\sum_{\sigma\in{\rm Gal}(H_{p^n}/K)}\chi(\sigma)d_n^\sigma\cdot\phi\left({\rm Log}^+(\kappa_1^+)\right)\\
&=\pm\phi(F_d)\cdot\phi\left({\rm Log}^+(\kappa_1^+)\right),
\end{align*}
and the result follows.
\end{proof}
\subsection{Proof for Lower Bound}
Before continuing with the proof let us briefly discuss the Selmer conditions at $\ell|N$. By our assumption that $N$ is square-free, $E$ has multiplicative reduction at $\ell$. Then we have the following easy facts (which can be deduced, for example from \cite[Section 1.3]{Rubin}). Let $v|\ell$ be a prime in $\mathcal{K}$ or $\mathcal{K}_m$ then
$$H^1(G_v, V)=0, H^1_\mathcal{F}(G_v, T)=H^1(G_v, T), H^1_\mathcal{F}(G_v, A)=0.$$
Starting from now we will use $\Lambda_{n,n'}$ for $\Lambda/(\omega_n((1+T)^{p^{t+a}}-1), p^{n'})$ and $T_{n,n'}=\mathbf{T}\otimes\Lambda_{n,n'}$. We also define $\Lambda_{n,n'}^+:=\Lambda/(\omega_n^+((1+T)^{p^{t+a}}-1), p^{n'})$.
\begin{proposition}
We have $H_{str}^1(\mathcal{K},\mathbf{T})=0$, and $X_{str}$ is torsion.
\end{proposition}
\begin{proof}
We have seen in the last section that $H_+^1(\mathcal{K},\mathbf{T})$ is free of rank one over $\Lambda$ and then the argument is completely the same as \cite[Lemma 3.4, Corollary 3.6]{WAN}.
\end{proof}
We write $H^1_{+,\mathrm{rel}}(\mathcal{K}, \mathbf{T})$ for the Selmer group whose local restrictions at primes not dividing $p$ the same as $\mathcal{F}^+$ and at $p=v\bar{v}$ is $+$ at $v$ and no restriction at $\bar{v}$.
\begin{lemma}
We have an isomorphism
$$H_+^1(\mathcal{K},\mathbf{T})\simeq H^1_{+,\mathrm{rel}}(\mathcal{K},\mathbf{T}).$$
\end{lemma}
\begin{proof}
Same as \cite[Lemma 3.7]{WAN}. The surjectivity follows from that the cokernel is embedded to $H^1(\mathcal{K}_{\bar{v}_0}, \mathbf{T})/H^1_+(\mathcal{K}_{\bar{v}}, \mathbf{T})$, which is torsion free.
\end{proof}
\begin{lemma}
We have $H_{\bar{v}}^1(\mathcal{K},\mathbf{T})$ is $0$.
\end{lemma}
\begin{proof}
Same as \cite[Lemma 3.8]{WAN}.
\end{proof}
The following proposition is the key of the whole argument. This is an analogue of \cite[Proposition 3.7]{WAN}. But we have to use a different argument in the $\pm$ setting.
\begin{definition}
We define $\mathrm{Sel}^+:=H^1_+(\mathcal{K}, \mathbf{A})$ and write $\mathrm{Sel}^+_{\mathrm{cofree}}$ for the co-torsionfree part of it.
\end{definition}
\begin{proposition}\label{3.9}
Consider the map
$$\Lambda\simeq \varinjlim_n(E^+(\mathcal{K}_{n,v})\otimes\mathbb{Q}_p/\mathbb{Z}_p)^*\rightarrow (\mathrm{Sel}_{\mathrm{cofree}}^+)^*.$$
which is the Pontryagin dual of the localization map.
Then for any height one prime $P$ of $\Lambda$ as above, we localize the above map at $P$ and write this $\Lambda_P$-module map as $j_{P,v}$. We also define
$$f_{v,P}=\exp_P(\mathrm{coker}\{H^1_+(\mathcal{K}, \mathbf{T})_P\rightarrow H^1_+(\mathcal{K}_v, \mathbf{T})_P\}.$$ Then:
$$\mathrm{exp}_P(\mathrm{coker}j_{P^\iota,v})=f_{v,P}.$$
We can define $j_{P,\bar{v}}$ similarly and have $\mathrm{exp}_P(\mathrm{coker}j_{P^\iota,\bar{v}})=f_{\bar{v},P}$.
\end{proposition}
\begin{proof}
For any height one prime $P\not=(p)$ generated by $g_P$. As in \cite{WAN} we take $g_\mathfrak{D}=g_P+p^m$ for $m>>0$ which generates another height one prime $\mathfrak{D}$. Recall both $H_+^1(\mathcal{K},\mathbf{T})$ and $H_+^1(\mathcal{K}_v,\mathbf{T})$ ($p=v\bar{v}$) are free rank one modules over $\Lambda$. Let $f$ be the image of $1$ under the non-zero map $\Lambda\rightarrow\Lambda$ upon the above identification. Note that for all but finitely many $m$ we have $\sharp (\Lambda/(f,g_\mathfrak{D}))<\infty$. Fix $\mathfrak{D}$, we choose $n>>0$ such that $\omega_n^+((1+T)^{p^{t+a}}-1)\subseteq (f,g_\mathfrak{D})$. Take $p^{n'}\geq \sharp(\Lambda/(f,g_\mathfrak{D}))$. Let
$$\mathcal{A}_{n,n'}^+:=\mathrm{Sel}_{\mathrm{cofree}}^+[\omega_n^+((1+T)^{p^{t+a}}-1)][p^{n'}].$$
Then by the divisibility of elements in $\mathrm{Sel}^+_{\mathrm{cofree}}$ for $x\in\mathcal{A}_{n,n'}$ and any $\tilde{n}\geq n,\tilde{n}'\geq n'$ we can find $y_{\tilde{n},\tilde{n}'}\in\mathrm{Sel}^+_{\mathrm{cofree}}[\omega_{\tilde{n}}((1+T)^{p^{t+a}}-1)][p^{\tilde{n}'}]$, such that $$\frac{\omega_{\tilde{n}}((1+T)^{p^{t+a}}-1)}{\omega_n^+((1+T)^{p^{t+a}}-1)}\cdot p^{\tilde{n}'-n'}\cdot y_{\tilde{n}, \tilde{n}'}=x_{n,n'}$$
and it is possible to make $y_{\tilde{n},\tilde{n}'}$'s a compatible projective system.\\

\noindent We discuss the local properties for these classes. Let $x_{\tilde{n},\tilde{n}'}=\omega_{\tilde{n}}^-((1+T)^{p^{t+a}}-1)\cdot y_{\tilde{n},\tilde{n}'}$. So $x_{\tilde{n},\tilde{n}'}\in\mathrm{Sel}_{\mathrm{cofree}}^+[\omega_{\tilde{n}}^+((1+T)^{p^{t+a}}-1)][p^{\tilde{n}'}]$. We claim that for $v|p$, $\mathrm{loc}_vx_{\tilde{n},\tilde{n}'}\in E^+(\mathcal{K})_{\tilde{n},v}/p^{\tilde{n}'}E^+(\mathcal{K}_{\tilde{n},v})$. This can be seen as follows. Recall we defined
$$\Lambda_m^+=\mathbb{Z}_p[[T]]/\omega_m^+((1+T)^{t+a}-1)$$
and proved $E^+(\mathcal{K}_{m,v})\simeq \Lambda_m^+$. Note that $\Lambda_m^+$ is a free $\mathbb{Z}_p$-module.
There are nature isomorphisms
$$(\varinjlim_{m'}\Lambda_{m'}^+)\otimes \frac{p^{-n'}\mathbb{Z}_p}{\mathbb{Z}_p}[\omega_m^+((1+T)^{p^{t+a}}-1)]\simeq \Lambda_m^+\otimes\frac{p^{-n'}\mathbb{Z}_p}{\mathbb{Z}_p}$$
The claim follows.\\

\noindent Now we claim that $\mathrm{loc}_vy_{\tilde{n},\tilde{n}'}\in H^1_+(\mathcal{K}_{\tilde{n},v}, T_{\tilde{n}'})$. Take $z\in E^+(\mathcal{K}_{\tilde{n},v})/p^{\tilde{n}'}E^+(\mathcal{K}_{\tilde{n},v})$. Then we know
$$\langle \omega_{\tilde{n}}^-((1+T)^{p^{t+a}}-1)y_{\tilde{n},\tilde{n}'}, z\rangle=0.$$
Since $\langle y_{\tilde{n},\tilde{n}'}, z\rangle_{\tilde{n}}=\mathrm{tr}_{\mathcal{K}_{\tilde{n},v}/\mathbb{Q}_p}\log_{\hat{E}} z\exp^*_{\omega_E}y_{\tilde{n},\tilde{n}'}$ by the description of Tate pairing in \cite[Page 5]{LZ}, we see that
$$\langle y_{\tilde{n},\tilde{n}'}, \omega^-_{\tilde{n}}((1+T)^{p^{t+a}}-1)z\rangle=0.$$
Notice that $\omega_{\tilde{n}}^+((1+T)^{p^a}-1)z=0$. Take $b$ such that $p^b\in(\omega^-_{\tilde{n}}((1+T)^{p^{t+a}}-1), \omega^+_{\tilde{n}}((1+T)^{p^{t+a}}-1))$. Then $\langle p^by_{\tilde{n},\tilde{n}'},z\rangle=0$.
Therefore
$$p^b\cdot\varprojlim_{\tilde{n}'}y_{\tilde{n},\tilde{n}'}\in H^1_+(\mathcal{K}_{\tilde{n},v}, T).$$
Then $\varprojlim_{\tilde{n}'}y_{\tilde{n},\tilde{n}'}\in H^1_+(\mathcal{K}_{\tilde{n},v}, T)$ since $\frac{H^1(\mathcal{K}_{\tilde{n},v}, T)}{H^1_+(\mathcal{K}_{\tilde{n},v}, T)}$ is torsion free.
In sum we have
$$\varprojlim_{\tilde{n},\tilde{n}'}y_{\tilde{n},\tilde{n}'}\in \varprojlim_{\tilde{n},\tilde{n}'}H_+^1(\mathcal{K}, \mathbf{T}_{\tilde{n},\tilde{n}'})=H_+^1(\mathcal{K},\mathbf{T})\simeq \Lambda.$$
On the other hand for any $y\in H^1_+(\mathcal{K}, \mathbf{T})$ let $y_{n,n'}$ be its image in $H^1_+(\mathcal{K}, T_{n,n'})$, then $\omega_n^-((1+T)^{p^{t+a}}-1)y_{n,n'}\in\mathrm{Sel}_{\mathrm{cofree}}^+$ by \cite[Lemma 1.3.8 (i)]{Rubin}. To see it is in the co-torsionfree part, first of all it is clearly $p$-power divisible. For any $g$ not a power of $p$, take an $\tilde{n}$ such that $\frac{\omega_{\tilde{n}}^+((1+T)^{p^{t+a}}-1)}{\omega_n^+((1+T)^{p^{t+a}}-1)}\in (p^{n'},g)$. Suppose $\frac{\omega_{\tilde{n}}^+((1+T)^{p^{t+a}}-1)}{\omega_n^+((1+T)^{p^{t+a}}-1)}\equiv w\cdot g(\mathrm{mod}\ p^{n'})$. Then $\omega_{\tilde{n}}^-((1+T)^{p^{t+a}}-1)y_{\tilde{n},n'}\cdot w\cdot g=\omega_n^-((1+T)^{p^{t+a}}-1)y_{n,n'}$. So $y_{n,n'}$  is $g$-divisible. We sum up the above discussion into
\begin{equation}\label{equation (3)}
\mathcal{A}_{n,n'}^+=\omega_n^-((1+T)^{p^{t+a}}-1)\mathrm{Im}(H^1_+(\mathcal{K}, \mathbf{T})\rightarrow H^1(\mathcal{K}, T_{n,n'}))
\end{equation}
Note that the kernel of $\mathrm{ker}(H^1_+(\mathcal{K},\mathbf{T})\rightarrow H^1_+(\mathcal{K}, T_{n,n'}))$ contains, but not necessarily equals $(\omega_n((1+T)^{p^{t+a}}-1),p^{n'})\cdot H^1_+(\mathcal{K}, \mathbf{T})$. Therefore we have a surjective map
\begin{equation}\label{equation (4)}
\Lambda_{n,n'}^+\simeq \omega^-_n((1+T)^{p^{t+a}}-1)\Lambda_{n,n'}\twoheadrightarrow \mathcal{A}_{n,n'}^+.
\end{equation}

\noindent We claim that $\frac{\sharp\{\mathrm{Sel}_{\mathrm{cofree}}^+[\omega_n^+((1+T)^{p^{t+a}}-1)][p^{n'}]\}}{\sharp(\Lambda_{n,n'}^+)}$
is a power of $p$ bounded independent of $m$. This can be seen as follows. By the structure theorem of $\Lambda$-modules the Pontryagin dual of $\mathrm{Sel}_{\mathrm{cofree}}^+$ can be embedded into $\Lambda$ with cokernel $M$ having finite cardinality. Then the claim follows by calculating $\mathrm{Tor}(M, \Lambda_{n,n'}^+)$ and showing that its cardinality is bounded by the square of the cardinality of $M$.
So if $\mathrm{ord}_Pf=d$, then $\sharp(\frac{\Lambda}{(f,\mathfrak{D})})\approx p^{md}$ where we use ``$\approx$'' to mean they differ by a power of $p$ bounded independent of $m$. So the kernel of (\ref{equation (4)}) has cardinality bounded independent of $m,n,n'$. We have by (\ref{equation (3)})
$$\sharp\mathrm{ker}\{\mathcal{A}_{n,n'}^+[\omega_n^+,\mathfrak{D}^\iota]\rightarrow \Lambda^*[\omega_n^+,\mathfrak{D}^\iota]\}\approx p^{md}$$
since the left hand side $\approx \mathrm{ker}\{\times f: \Lambda_{n,n'}^+[\mathfrak{D}^\iota]\rightarrow \Lambda_{n,n'}^+[\mathfrak{D}^\iota]\}$.
(Note that the Galois representation at $\mathbf{A}$ is obtained by the one for $\mathbf{T}$ composed with $\iota$).
On the other hand if $\mathrm{ord}_{P^\iota}(\mathrm{ker}\{\mathrm{Sel}_{\mathrm{cofree}}^+\rightarrow \Lambda^*\}^*)=d'$, then the
$$\sharp\mathrm{ker}\{\mathrm{Sel}_{\mathrm{cofree}}^+[\omega_n^+,\mathfrak{D}^\iota]\rightarrow \Lambda^*\}\approx p^{md'}.$$
Taking $m\rightarrow \infty$ we get $d=d'$ and the proposition for $P\not=(p)$.

If $P=(p)$ then we choose $\mathfrak{D}=(p+T^m)$ and argue similarly.
\end{proof}
Now we are ready to prove the main theorem in the introduction.
\begin{theorem}
There is a quasi-isomorphism $X^+\sim \Lambda \oplus X^+_{\mathrm{tor}}$. Moreover
\begin{itemize}
\item[(1)]
Under assumption (2) then for any height one prime $P$ of $\Lambda$,
$$\mathrm{lg}_PX_{\mathrm{tor}}^+=2\mathrm{lg}_P(\frac{H^1_+(\mathcal{K}, T\otimes\Lambda(\Psi))}{\Lambda\kappa_1^+}).$$
\item[(2)]
Under assumption (1) the above is true for all height one primes $P\not=(p)$
\end{itemize}
\end{theorem}
\begin{proof}
We only need to prove ``$\geq$''. The proof is the same as the proof of \cite[Theorem 3.14]{WAN} using two Poitou-Tate long exact sequences: one taking $\mathcal{G}^*=\mathcal{F}_v^+$ and $\mathcal{F}^*$ to be our $\mathcal{F}^+$, the other one taking $\mathcal{F}^*$ to be $v$ and $\mathcal{G}^*=\mathcal{F}_v^+$.
Applying the results we proved before (Propositions \ref{3.1}, \ref{3.2}, \ref{3.3}, \ref{3.9}, in particular we use Proposition \ref{3.9} in place of \cite[Proposition 3.9]{WAN}) we get the theorem.
\end{proof}
Finally we give the proof of Corollary \ref{Corollary} in the introduction.
\begin{proof}
A large part of the argument is the same as the proof of \cite[Theorem 1.7]{WAN} using $\mathcal{F}^+$ in the place of $\mathcal{F}$ in \emph{loc.cit}, except that in the control theorem for the $\mathcal{F}^+$ Selmer group under specialization from $\mathcal{K}_\infty^-$ to $\mathcal{K}$ we have to replace the local argument at primes above $p$. However this is the same as \cite[Proposition 9.2, Theorem 9.3]{Kobayashi}, replacing the cyclotomic $\mathbb{Z}_p$-extension by anticyclotomic $\mathbb{Z}_p$-extension. The proof is identical to the argument in \cite{Kobayashi}.
\end{proof}

\textsc{Francesc Castella, Department of Mathematics, University of California, Los Angeles, CA, 90095-1555, USA}\\
\indent \textit{E-mail Address}: castella@math.ucla.edu\\

\noindent\textsc{Xin Wan, Department of Mathematics, Columbia University, New York, NY, 10027, USA}\\
\indent \textit{E-mail Address}: xw2295@math.columbia.edu

\end{document}